\documentclass[11pt,draft,amsmath,amscd,amsbsy,amssymb]{amsart}
 \relax
\chardef\bslash=`\\ 

\makeatletter \def\verbatim{\interlinepenalty\@M \@verbatim
\leftskip\@totalleftmargin\advance\leftskip2pc
\frenchspacing\@vobeyspaces \@xverbatim} \makeatother \hfuzz1pc
\makeatletter \def\dgt@k{\dg@DX=-3 \dg@DY=2 \dg@SIZE=3}
\makeatother
\makeatletter \def\dgt@kk{\dg@DX=3 \dg@DY=-1 \dg@SIZE=3}
\theoremstyle{plain} \newtheorem{theo}{Theorem}[section]

\newtheorem{lemma}[theo]{Lemma}
\newtheorem{propos}[theo]{Proposition}
\theoremstyle{definition} 
\newtheorem{defin}[theo]{Definition} \newtheorem{que}[theo]{Question}

\usepackage[all]{xy}
\begin{document}
\title[Fuzzy Prokhorov metric on the set of probability measures]
{Fuzzy Prokhorov metric on the set of probability measures}

\author[D.~Repov\v s]{Du\v{s}an~Repov\v{s}}
\address{Faculty of Mathematics, Physics and Mechanics, and Faculty of Education,
University of Ljubljana, P.O.B. 2964, 1001 Ljubljana,  Slovenia}
\email{dusan.repovs@guest.arnes.si}

\author[A.~Savchenko]{Aleksandr Savchenko}
\address{Kherson Agrarian University, Kherson, Ukraine}
\email{savchenko1960@rambler.ru}

\author[M.~Zarichnyi]{Mykhailo Zarichnyi}
\address{Lviv National University,
79000 Lviv, Ukraine and Institute of Mathematics, Physics and Mechanics, P.O.B. 2964, 1001 Ljubljana,  Slovenia}

\email{mzar@litech.lviv.ua}
\thanks{}



\begin{abstract} We introduce a fuzzy metric on the set of probability  measures  on a fuzzy metric space.
The construction is an analogue, in the realm of fuzzy metric spaces, of the Prokhorov metric
on the set of probability measures on  compact metric spaces.

\end{abstract}
\date{\today}
\maketitle
\section{Introduction}

The notion of fuzzy metric space first appeared in
\cite{KM} and it was later modified in \cite{GV}. The version from
\cite{GV}, despite being more restrictive, determines the class of
spaces that are closely connected with the class of metrizable
topological spaces. This notion was widely used in various
papers 
devoted to  fuzzy topology and it has found numerous applications -- 
in particular  to color image processing (see e.g. \cite{GMS} and
the references therein).

Different notions and results of the theory of metric spaces have
their analogues for  fuzzy metric spaces. At the same time,
there are phenomena in the realm of fuzzy metric spaces that have
no immediate analogue for metric spaces. The completeness and
 existence of non-completable fuzzy metric spaces can serve as
an example. This demonstrates that the fuzzy metric seems to be a
structure that leads to a theory which appears to be richer 
than that of metric spaces.

In the theory of fuzzy metric spaces, there are analogues of various constructions from 
theory of metric spaces. In particular, a fuzzy Hausdorff metric was defined in \cite{RR}. The fuzzy metrics on (finite and countable) powers and $G$-symmetric powers were defined, in particular
in \cite{RN,Sav}.

In this paper, we consider a fuzzy analogue of the Prokhorov metric defined on the set of all probability measures of a compact fuzzy metric space. We prove that this metric induces the weak* convergence of probability measures on compact metrizable space. Our aim of studying the space of probability measures on fuzzy metric spaces is twofold. First, we use the property of spaces of probability measures to be absolute extensors and this allows us to solve the problem of (continuous) extension of fuzzy metric defined on a closed subspace. 

Note that the problem of extension of structures is of a fundamental character and
it arises in various
areas of mathematics: differential equations (extensions of solutions), functional analysis (extensions of functionals), topology (extensions of continuous maps), etc. 
The classical problem of extensions of metrics, first solved by Hausdorff in the 1930's, is also of this type.

Our theorem on continuous extension of fuzzy metrics relies on the extensional properties of the spaces of probability measures. These properties are known to hold
only for the class of metrizable spaces. Thus the reason why we are dealing with the fuzzy metric spaces in the sense of \cite{GV} is the fact that the topology induced by any fuzzy metric in this sense is metrizable (see \cite{GR}).

Our second reason to investigate the spaces of probability measures is their applicability
to the theory of probabilistic systems (see e.g. \cite{dVR, B,  vBW}). In the fuzzy metric setting, we come to the problem of fuzzy metrization of the sets of probability measures.
Note also that a fuzzy ultrametric on the set of probability measures with compact supports on fuzzy ultrametric spaces was defined in \cite{SaZa}.

One of the main results is that the construction of the Prokhorov metric determines a functor in the category of fuzzy metric spaces and nonexpanding maps. As was remarked above, we also apply the construction of Prokhorov metric to the problem of extensions of fuzzy metrics.

\section{Preliminaries}

\subsection{Fuzzy metric spaces}
The notion of fuzzy metric space, in one of its forms, is
introduced by  Kramosil and Michalek \cite{KM}. In the present
paper we use the version of this concept given by George and Veeramani
\cite{GV}. We start with some necessary definitions.
\begin{defin} A binary operation $* \colon [0, 1]\times[0, 1] \to [0, 1]$ is a continuous
t-norm if $*$  satisfies the following conditions:
\begin{enumerate}
\item[(i)] $*$ is commutative and associative;
\item[(ii)] $*$ is continuous;
\item[(iii)] $a *1 = a$ for all $a \in [0, 1]$; and
\item[(iv)] $a * b \le c * d$ whenever $a\le c$ and $b \le d$, and $a, b, c, d \in[0, 1]$.
\end{enumerate}
\end{defin}

The following are examples of t-norms: $a*b=ab$; $a*b=\min\{a,b\}$.

\begin{defin}\label{o5} A 3-tuple $(X,M, *)$ is said to be a fuzzy metric space
if $X$ is an arbitrary set, $*$ is a continuous t-norm and $M$ is a fuzzy set on
$X^2 \times (0,\infty)$ satisfying the following conditions for all $x, y, z \in X$ and $s, t > 0$:
\begin{enumerate}
\item[(i)] $M(x, y, t) > 0$,
\item[(ii)] $M(x, y, t) = 1$ if and only if $x = y$,
\item[(iii)] $M(x, y, t) = M(y, x, t)$,
\item[(iv)] $M(x, y, t) \ast M(y, z, s) \le M(x, z, t + s)$,
\item[(v)] the function $M(x, y, -) \colon (0,\infty) \to (0, 1]$ is continuous.
\end{enumerate}
\end{defin}

It was proved in \cite{GV} that in a fuzzy metric space $X$, the function  $M(x, y, -)$ is non-decreasing for all $x, y \in X$.
The following notion was introduced in \cite{GV} (see Definition 2.6 therein).

\begin{defin}
Let $(X,M, *)$ be a fuzzy metric space and let $r \in (0, 1)$,
$t > 0$ and $x \in X$. The set
$$B(x, r, t) = \{y \in X \mid M(x, y, t) > 1 - r\}$$
is called the {\em open ball} with center $x$ and radius $r$ with respect to $t$.
\end{defin}

The family of all open balls in a fuzzy metric space $(X,M, *)$ forms a base of a topology in $X$; this topology is denoted by $\tau_M$
and is known to be metrizable (see \cite{GV}). In the sequel, if we speak on a fuzzy metric on a topological space, we assume that this metric generates the  initial topology on the space.

Note that $B(x,r,t_1)\subset B(x,r,t_2)$, whenever $t_1\le t_2$.

If $(X,M, *)$ is a fuzzy metric space and $Y\subset X$, then, clearly, $$M_Y=M|(Y\times Y\times(0,\infty))\colon Y\times Y\times(0,\infty)\to [0,1]$$
is a fuzzy metric on the set $Y$. We say that the fuzzy metric $M_Y$ is {\em induced} on $Y$ by $M$.

If $(X,M, *)$ is a fuzzy metric space, then the family $$\{U_n=\{(x,y)\in X\times X\mid M(x,y,\frac{1}{n})>1-\frac{1}{n}\}\mid n\in\mathbb N\}$$ is a base of a uniform structure on $X$. This uniform structure is known to generate the topology $\tau_M$ on $X$.

Let $(X,M,*)$ and $(X',M',*)$ be fuzzy metric spaces. A map $f\colon X\to X'$ is called {\em nonexpanding} if $M'(f(x),f(y),t)\ge M(x,y,t)$, for all
$x,y\in X$ and $t>0$. For our purposes, it is sufficient to consider the class of fuzzy metric spaces with the same fixed norm (e.g.  $*$).

The fuzzy metric spaces (with the norm $*$) and nonexpanding maps form a category, which we denote by $\mathcal{FMS}(*)$. By $\mathcal{CFMS}(*)$ we denote its subcategory consisting of compact fuzzy metric spaces.

\subsection{Spaces of probability measures} Let $X$ be a metrizable space.  We denote
the space of probability measures with compact
support  in $X$ by $P(X)$ (see  e.g.  \cite{Par} for the necessary definitions that concern probability measures). Recall that the {\em support} of a
probability measure $\mu\in P(X)$ is the minimal (with respect to the inclusion) closed set $\mathrm{supp}(\mu)$ such that
$\mu(X\setminus \mathrm{supp}(\mu))=0$. For any $x\in X$, by $\delta_x$ we denote the Dirac measure concentrated at $x$.

Any probability measure $\mu$ of finite support can be represented as follows: $\mu=\Sigma_{i=1}^n\alpha_i\delta_{x_i}$, where
$\alpha_1,\dots,\alpha_n\ge0$ and $\Sigma_{i=1}^n\alpha_i=1$. By $P_\omega(X)$ we denote the set of
all probability measures with finite supports in $X$.

The set $P(X)$ is endowed with the weak* topology, i.e., the topology induced by the weak* convergence. A sequence $(\mu_i)$ in $P(X)$ weakly* converges to $\mu\in P(X)$ if $\lim_{i\to\infty}\int_X\varphi d\mu_i= \int_X\varphi d\mu$, for every $\varphi\in C(X)$. Equivalently, $\overline{\lim}_{i\to\infty}\mu_i(C)\le \mu(C)$, for every closed subset $C$ of $X$.

If $X$ is a compact Hausdorff space, then there exists a natural map $\psi_X\colon P^2(X)=P(P(X))$ defined by the formula:
$$\int_{X}\varphi d\psi_X(M)= \int_{P(X)}\bar\varphi d M,$$
where $\bar\varphi\colon P(X)\to\mathbb R$ is defined by the formula:  $\bar\varphi(\mu)= \int_{X}\varphi d\mu$.

Let $X,Y$ be metrizable spaces. Every continuous map $f\colon X\to Y$  generates a map $P(f)\colon P(X)\to P(Y)$ defined by the following condition: $P(f)(\mu)(A)=\mu (f^{-1}(A))$, for every Borel subset of $Y$. It is known that the map $P(f)$ is also continuous.

\section{\L ukasiewicz norm and fuzzy metric on the set of probability measures}

Recall that the \L ukasiewicz t-norm  is defined by the formula $$x\ast y=\max\{x+y-1,0\},\ x,y\in[0,1].$$ In the sequel, $\ast$ stands for the \L ukasiewicz t-norm.

In the sequel, let $(X,M,\ast)$ be a compact fuzzy metric space. For every $A\subset X$, $r\in (0,1)$, $t\in(0,\infty)$ define:
$$A^{r,t}=\cup\{B(x,r,t)\mid x\in A\}\subset X.$$

\begin{lemma} For every $A\subset X$, every $r,\varrho\in(0,1)$ such that $r\ast\varrho\in(0,1)$, and every $t,s\in (0,\infty)$ we have:
$$(A^{r,t})^{\varrho,s}\subset A^{r+\varrho,t+s}.$$

\end{lemma}

\begin{proof} Let $x\in (A^{r,t})^{\varrho,s}$, then there exist $y,z\in X$ such that $z\in A$, $M(z,y,t)>1-r$, $M(y,x,s)>1-\varrho$. Whence \begin{align*}M(x,z,t+s)&\ge M(x,y,s)\ast M(y,z,t) \\ &=\max\{M(x,y,s)+M(y,z,t)+1,0\}  >\max\{1-r+1-\varrho-1,0\}\\ &= 1-r-\varrho,\end{align*} i.e. $x\in A^{r+\varrho,t+s}$.
\end{proof}
The following is obvious.

\begin{lemma} If $t_1\le t_2$, then, for every $A\subset X$, we have  $A^{r,t_1}\subset A^{r,t_2}$.
\end{lemma}

\begin{lemma} Let $t_0\in(0,\infty)$. For every $\varepsilon>0$, there exists $\eta>0$ such that, for every $x,y\in X$, if $|t-t_0|<\eta$, then $|M(x,y,t)-M(x,y,t_0)|<\varepsilon$.  \end{lemma}

\begin{proof} Let $J$ be a closed interval such that $t_0$ is its interior point. Since the map $M\colon X\times X\times (0,\infty)$ is continuous (see \cite[Proposition 1]{RR}), its restriction onto $X\times X\times J$ is uniformly continuous and the result follows.
\end{proof}
\begin{lemma}\label{l:unif} Let $t_0\in (0,\infty)$, $r_0\in (0,1)$. For every $\varepsilon>0$, there exists $\eta>0$ such that, $B(x,r_0+\varepsilon,t)\supset B(x,r_0,t_0)$, whenever $x\in X$ and $|t-t_0|<\eta$.
\end{lemma}
\begin{proof} This is a consequence of the previous lemma.
\end{proof}

Define the function $\hat M\colon P(X)\times P(X)\times (0,\infty)\to[0,1]$ by the formula
\begin{align*} \hat M(\mu,\nu, t)=&1-\inf\{r\in(0,1)\mid \mu(A)\le\nu(A^{r,t})+r \text{ and } \nu(A)\le\mu(A^{r,t})+r\\ &\text{for every Borel subset } A\subset X\}. \end{align*}
We will see later that the function $\hat M$ is well defined.

\begin{theo} The function $\hat M$ is a fuzzy metric on the set $P(X)$.
\end{theo}

\begin{proof} We are going to verify the properties from the definition of the fuzzy metric.

First, remark that $\hat M(\mu,\nu,t)>0$. Indeed, since the set $\mathrm{supp}(\mu)\cup \mathrm{supp}(\nu)$ is compact, there exists $r>0$ such that, for any $x\in  \mathrm{supp}(\mu)\cup \mathrm{supp}(\nu)$, we have $B(x,r,t)\supset \mathrm{supp}(\mu)\cup \mathrm{supp}(\nu)$ (see \cite{RR}). Then, for any Borel set $A\subset X$, such that $\mu(A)>0$, we obtain $A^{r,t}\supset \mathrm{supp}(\nu)$ and therefore $\nu ( A^{r,t})=1$. Then clearly, $\hat M(\mu,\nu,t)>1-r>0$. Note that at the same time we have proven that the function $\hat M$ is well defined.

Clearly, $\hat M(\mu,\mu,t)=1$. Let now $\hat M(\mu,\nu,t)=1$. Then it is easy to see that, for any $r\in(0,1)$ and $t>0$, we see that $$\mu(A)\le \nu(A^{r,t})+r,\ \nu(A)\le \mu(A^{r,t})+r,$$ whence $\mu(A)=\nu(A)$, for any Borel $A$.  From the elementary properties of probability measures it follows that  $\mu(A)=\nu(A)$.

Clearly, $\hat M(\mu,\nu,t)=\hat M(\nu,\mu,t)$.

Let $\mu,\nu,\tau\in P(X)$. Suppose that $\hat M(\mu,\nu,t)>a$, $\hat M(\nu,\tau,s)>b$ for some $a,b\in(0,1)$. Then there exist $r\in(0,1-a)$, $\varrho\in(0,1-b)$ such that
\begin{align*} \mu(A^{r,t})\le\nu(A^{r,t})+r,\ \nu(A^{r,t})\le\mu(A^{r,t})+r,\\  \nu(A^{\varrho,s})\le\tau(A^{\varrho,s})+\varrho,\ \tau(A^{\varrho,s})\le\mu(A^{\varrho,s})+\varrho,\end{align*} for every Borel subset $A\subset X$.

Then $$\mu(A)\le \nu(A^{r,t})+r\le \tau((A^{r,t})^{\varrho,s})+r+\varrho\le \tau(A^{r+\varrho,t+s})+r+\varrho,$$
$$\tau(A)\le \nu(A^{r,t})+r\le \mu((A^{r,t})^{\varrho,s})+r+\varrho\le \mu(A^{r+\varrho,t+s})+r+\varrho,$$
hence it follows that $$\hat M(\mu,\tau, r+s)\ge 1-(r+\varrho)=1-(r+\varrho)=1-r+1-\varrho -1=a+b-1=a\ast b.$$ This proves property (iv) from the definition of the fuzzy metric.

It remains to prove that, for every $\mu,\nu\in P(X)$, the map $$t\mapsto \hat M(\mu,\nu,t)\colon (0,\infty)\to[0,1]$$ is continuous. First note that this map is nondecreasing. Indeed, suppose that $\hat M(\mu,\nu,t_1)=1-r_1$ and $t_1\le t_2$. Then, for every $r>r_1$, and every Borel subset $A$ of $X$, we have $$\mu(A)\le \nu(A^{r,t_1})+r\le \nu(A^{r,t_1})+r,\ \nu(A)\le \mu(A^{r,t_1})+r\le \mu(A^{r,t_1})+r,$$ whence $\hat M(\mu,\nu,t_2)=1-\inf\{r\mid r>r_1\}\ge 1-r_1=\hat M(\mu,\nu,t_1)$.

Suppose now that $\hat M(\mu,\nu,t_0)>1-r_0$. Then there is $r<r_0$ such that $$\mu(A)\le \nu(A^{r,t_0})+r,\ \nu(A)\le \mu(A^{r,t_0})+r,$$
for every Borel subset $A$ of $X$.

Let $\varepsilon>0$ be such that $r+\varepsilon<r_0$. There is $\eta>0$ such that, for every $t\in(t_0-\eta,t_0)$ and every $x\in X$,  we have $B(x,r+\varepsilon,t)\supset B(x,r,t_0)$. Then, for every Borel subset $A$ of $X$, we have $$\mu(A)\le \nu(A^{r+\varepsilon,t})+r+\varepsilon,\ \nu(A)\le \mu(A^{r+\varepsilon,t_0})+r+\varepsilon,$$
whence $$1-\hat M(\mu,\nu,t)\ge 1- (r+\varepsilon)> 1-r_0.$$ This proves the left-continuity of the function $M(\mu,\nu,-)$.

To prove the right-continuity at $t_0$, let $\varepsilon>0$.  By Lemma \ref{l:unif}, there exists $\eta>0$ such that, for every Borel subset $A$ of $X$ and every $r\in(0,1)$, we have $A^{r+\varepsilon,t_0}\supset A^{r,t}$, whenever $|t_0-t|<\eta$.

Let $t\in (t_0,t_0+\eta)$. Suppose that $$\hat M(\mu,\nu,t_0)=1-r_0,\ \hat M(\mu,\nu,t)=1-r.$$  By the definition of $\hat M$, there exists $r'\in(r,r+\varepsilon)$ such that for every Borel subset $A$ of $X$, we have $$\mu(A)\le \nu(A^{r',t})+r',\ \nu(A)\le \mu(A^{r',t})+r'.$$ Then $$\mu(A)\le \nu(A^{r'+\varepsilon,t})+r',\ \nu(A)\le \mu(A^{r'+\varepsilon,t})+r'$$ and therefore $r_0\le r'+\varepsilon\le r+2\varepsilon$. Since $\hat M(\mu,\nu,t)$ is nondecreasing, we conclude that $r\le r_0$, whence $|r-r_0|<2\varepsilon$.
\end{proof}

\begin{propos} Let a sequence $(\mu_i)$ in $P(X)$ converge to $\mu\in P(X)$ with respect to the topology induced by the fuzzy metric $\hat M$. Then $(\mu_i)$ weakly* converges to $\mu$.
\end{propos}
\begin{proof} There exist $(r_i,t_i)\in (0,1)\times(0,\infty)$, $i\in\mathbb N$, such that the family $$\{\hat U_i=\{(\mu,\nu)\in P(X)\times P(X)\mid \hat M(\mu,\nu,t_i)>1-r_i\}\mid i\in\mathbb N\}$$ forms a countable decreasing base of the uniform structure in $P(X)$ generated by the fuzzy metric $\hat M$. Without loss the generality, one may assume that the family $$\{ U_i=\{(x,y)\in X\times X\mid  M(x,y,t_i)>1-r_i\}\mid i\in\mathbb N\}$$ forms a decreasing base of the uniform structure in $X$ generated by the fuzzy metric $ M$ (if necessary, we decrease $t_i$ and/or  $r_i$). We also assume that $r_i\to 0$ whenever $i\to\infty$.

Let $A$ be a Borel subset of $X$. Then $$\mu_i(A)\le \mu(A^{r_i,t_i})+r_i,\  \mu(A)\le \mu_i(A^{r_i,t_i})+r_i $$ and therefore $$\overline{\lim}_{i\to\infty}\mu_i(A)\le  \overline{\lim}_{i\to\infty}\mu(A^{r_i,t_i})+r_i=\lim_{i\to\infty}\mu(A^{r_i,t_i})=\mu(\bar A),$$
whence $\overline{\lim}_{i\to\infty}\mu_i(C)\le \mu(C)$, for any closed $C\subset X$. Hence, $(\mu_i)$ weakly* converges to $\mu$.

\end{proof}

\begin{lemma}\label{l:balls} For every $\delta\in (0,1)$ and every $t>0$ there exists a countable family open balls $\{B(x_i,r_i,t)\mid i\in\mathbb N)\}$ such that the following are satisfied:
\begin{enumerate}
\item $r_i<\delta$ for all $i$;
\item $\cup_{i=1}^\infty B(x_i,r_i,t)=X$;
\item $\mu(\partial B(x_i,r_i,t))=0$ for all $i$.
\end{enumerate}
\end{lemma}
\begin{proof} Let $D=\{x_i\mid i\in\mathbb N\}$ be a countable dense subset of $X$. For every $x\in D$, $r\in(0,1)$, $t\in(0,\infty)$, let $S(x,r,t)=\{y\in X\mid M(x,y,t)=1-r\}$. Then, because of the continuity of $M$, we see that $S(x,r,t)\subset\partial B(x,r,t)$.

Since the family $\{S(x_i,r,t)\mid r\in(\delta/2,\delta)\}$ is disjoint, we see that there exists $r_i\in  (\delta/2,\delta)$ such that $\mu (S(x_i,r_i,t))=0$. Then also $\mu(\partial B(x_i,r_i,t))=0$.

Note that, for every $t_0\in (0,\infty)$ the family $$\{B(x,r,t')\mid x\in X,\ r\in(0,\delta/2),\
t'<t \}$$ forms a base of the topology in $X$. Thus, since $D$ is dense in $X$, for any $x\in X$ there is $i\in\mathbb N$, $r\in(0,\delta/2)$, $t'<t$ such that $x_i\in B(x,r,t')$. Then also $x\in B(x_i,r,t')\subset B(x_i,r_i,t)$. Therefore, (2) holds.

\end{proof}

\begin{theo} Suppose that $(\mu_i)$ weak* converges to $\mu$, for $\mu,\mu_i\in P(X)$. Then $(\mu_i)$  converges to $\mu$ in the topology induced by the fuzzy metric $\hat M$.
\end{theo}
\begin{proof} Let $t,\varepsilon>0$.

We want to show that there exists $N\in\mathbb N$ such that, for every $i\ge N$, $\hat M(\mu_i,\mu,t)<\varepsilon$.
The latter means that $$\mu(B)\le \mu_i(B^{\varepsilon,t})+\varepsilon,\ \mu_i(B)\le \mu(B^{\varepsilon,t})+\varepsilon,$$ for every Borel subset $A$  of $X$.

Let $\delta\in(0,\varepsilon/3)$. By Lemma \ref{l:balls}, there exists a collection of open balls $\{B_i= B(x_i,r_i,t/2)\mid i\in\mathbb N\}$ such that $r_i\in(0,\delta/2)$ such that $\cup_{j=1}^\infty B_j=X$ and $\mu(\partial B_j)=0$, for every $j$.

Clearly, there exists $k\in  \mathbb N$ such that $\mu(\cup_{j=1}^k B_j)>1-\delta$.

Consider the family $$\mathcal A=\left\{\cup_{j\in J}B_j\mid J\subset\{1,\dots,k\}\right\}.$$
Note that, for every $A\in\mathcal A$, since $\partial A\subset \cup_{j=1}^k\partial B_j$, we conclude  that $\mu(\partial A)=0$. Since $A$ is open and $(\mu_i)$ weakly* converges  to    $\mu$, we see that   $\lim_{i\to\infty}\mu_i(A)=\mu(A)$.

There exists $N\in\mathbb N$ such that  $|\mu_i(A)-\mu(A)|<\delta$, for all $i\ge N$ and for all $A\in\mathcal A$. Note that then $\mu_i(\cup_{j=1}^kB_j)\ge 1-2\delta$, for all $i\ge N$.

Given a Borel set $B$ in $X$, let $$A=\{B_j\mid B_j\cap A\neq\emptyset,\ j=1,\dots,k \}.$$

Then the following holds:
\begin{enumerate}
\item $A\subset B^{\delta,t}$;
\item $B\subset A\cup \left(X\setminus \cup_{j=1}^kB_j\right)$;
\item $|\mu_i(A)-\mu(A)|<\delta$;
\item $\mu \left(X\setminus \cup_{j=1}^k)B_j\right)\le \delta$, $\mu_i \left(X\setminus \cup_{j=1}^k)B_j\right)\le 2\delta$, for all $i>N$.
\end{enumerate}

Indeed, one has only to check (1). Suppose that $x\in A$,
then $x\in B_j=B(x_j,r_j,t/2)$, for some $j$, and there exists $y\in B_j\cap B$. We obtain $$M(x,y,t)\ge M(x,x_j,t/2)\ast M(x_j,y,t/2)>(1-r)\ast(1-r)\ge 1-2r,$$
whence $x\in B^{\delta,t}$.

Therefore, for every $i\ge N$, we have
\begin{align*} \mu(B)\le \mu(A)+\delta\le \mu_i(A)+2\delta \le \mu_i(B^{\delta,t})+2\delta\le \mu_i(B^{\varepsilon,t})+\varepsilon;\\
\mu_i(B)\le \mu_i(A)+2\delta\le \mu(A)+3\delta \le \mu_i(B^{\delta,t})+3\delta\le \mu_i(B^{\varepsilon,t})+\varepsilon.
\end{align*}

Since $B$ is an arbitrary Borel set, we conclude that $\hat M(\mu,\mu_i,t)\ge1-\varepsilon$, for all $i\ge N$.

\end{proof}

\begin{propos}
Let $f\colon X\to Y$ be a nonexpanding map of compact fuzzy metric spaces $(X,M,\ast)$ and $X',M',\ast)$. Then the induced map $P(f)\colon P(X)\to P(Y)$ is also nonexpanding.
\end{propos}
\begin{proof}  Note that, since the map $f$ is nonexpanding, $B(x,r,t)\subset f^{-1}(B'(f(x),r,t))$, for every $x\in X$, $r\in(0,1)$ and $t\in(0,\infty)$. Therefore, for any $A\subset X'$ we have $f^{-1}(A)^{r,t}\subset f^{-1}(A^{r,t})$.

Let $\mu,\nu\in P(X)$ and let $A$ be a Borel subset of $X'$. If $\hat M(\mu,\nu,t)>1-r$, then $$P(f)(\mu)(A) =\mu(f^{-1}(A))\le \nu(f^{-1}(A)^{r,t})+r\le \nu(f^{-1}(A^{r,t}))+r=P(f)(\nu)(A^{r,t})+r $$ and similarly $P(f)(\nu)(A)\le P(f)(\mu)(A^{r,t})+r $, whence  $\hat M'(P(f)(\mu),P(f)(\nu),t)>1-r$ and the map $P(f)$ is nonexpanding.

\end{proof}

Therefore we obtain the probability measure functor $P$ acting in the category $\mathcal{CFMS}(*)$.

\begin{propos}\label{p:iso} For any compact fuzzy metric space $X$, the map $x\mapsto \delta_x\colon
 X\to P( X) $ is an isometric embedding.
\end{propos}
\begin{proof} Let $x,y\in X$, $t\in(0,\infty)$ and $M(x,y,t)=1-r_0$. Note that then $y\notin B(x,r_0,t)=\{x\}^{r_0,t}$. Then for every $r\in(0,1)$, $r<r_0$, we have $y\notin B(x,r,t)$, whence $$\delta_x(\{x\}) =1>\delta_y(\{x\}^{r,t})+r=\delta_y(B(x,r,t))+r =0+r=r.$$ Therefore, $\hat M(\delta_x,\delta_y,t)\ge 1-r_0$.

Let $r\in(0,1)$, $r>r_0$. If $A\neq\emptyset$ is a Borel subset of $X$ with $x\in A$, then $y\in B(x,r,t)\subset A^{r,t}$ and $$1=\delta_x(A)\le \delta_y(A^{r,t})+r.$$ Similarly, $\delta_y(A)\le \delta_x(A^{r,t})+r$. We conclude that $\hat M(\delta_x,\delta_y,t)\le 1-r_0$.

\end{proof}

Thus, the identity functor on the category $\mathcal{CFMS}(*)$ is a subfunctor of the probability measure functor $P$.

\section{Extension of fuzzy metrics}\label{s:ext}

In this section we provide an application of the Prokhorov fuzzy metric to the problem of extensions of fuzzy metrics.
\begin{theo} Let $Y$ be a nonempty closed subset of a compact metrizable space $X$. Then any fuzzy metric on $Y$ can be extended over $X$.
\end{theo}
\begin{proof} Let $M$ be a fuzzy metric on $Y$. without loss of generality, one may assume that $Y$ is infinite (otherwise, one can multiply $Y$ by an infinite compact fuzzy metric space, say $[0,1]$ and attach $X$ to $Y\times[0,1]$ by identifying every $y\in Y$ and $(y,0)\in Y\times[0,1]$.)

Then the space $P(Y)$ is homeomorphic to the Hilbert cube  and from the results of infinite-dimensional topology of the Hilbert cube it easily follows that there exists an embedding $F\colon X\to P(Y)$ such that $F(y)=\delta_y$, for every $y\in Y$ (see  e.g.  \cite{F}).

Define $M'\colon X\times X\times(0,\infty)\to\mathbb R$ by the formula $M'(x,y,t)=\hat M(F(x),F(y),t)$. It follows from Proposition \ref{p:iso} that $M'$ is an extension of $M$.
\end{proof}

\section*{Epilogue}

A measure $\mu$ on $X$ is said to be a {\em subprobability measure } if there is a probability measure $\mu'$ on $X$ such that $\mu(A)\le\mu'(A)$, for every Borel set $A$ of $X$. The set of all subprobability measures on $X$ can be identified with the subspace of $X\cup\{1\}$, where $1=\{0\}$ stands for a terminal object in the category of metrizable spaces. Given a fuzzy metric $M$ on a compact metrizable space $X$, there exists the unique fuzzy metric $$M'\colon (X\cup\{1\})\times (X\cup\{1\})\times (0,\infty)\to\mathbb R$$ that extends $M$ and such that  $M'(x,0,t)=
\frac12$, for every $x\in X$. (Indeed, since, for every $x,y\in X$ and every $t,s\in(0,\infty)$, we have $$M(x,y,t )\ast M(y,0,s)= M(x,y,t )\ast \frac12=\max \{M(x,y,t )-\frac12,0\}\le \frac12= M(x,0,t+s),$$ the fuzzy metric $M'$ is well-defined.)

The set $P'(X)$ of subprobability measures on $X$ then can be interpreted as the set $P(X\cup\{0\})$.

The following questions remains open.

\begin{que} Are there analogues of the above results for fuzzy metric spaces which are not necessarily compact?
\end{que}

\begin{que} Is there a fuzzy analogue of the Kantorovich metric in the set of probability measures?
\end{que}

See e.g.  \cite{R} for the definitions and properties of the Kantorovich metric.
\begin{que} Is there a fuzzy analogue of the Prokhorov metric in the set of probability measures for another choice of t-norm?
\end{que}

\begin{que} Is the canonical map $\psi\colon (P^2(X),\hat{\hat M})\to (P(X),\hat M)$ nonexpanding?
\end{que}

The corresponding question for the metric spaces and nonexpanding maps is discussed in \cite{B}.

\section*{Acknowledgements}

This research was supported by the
Slovenian Research Agency grants P1-0292-0101 and
M-2057-0101.
 We thank the referees for comments and suggestions.

\end{document}